\documentclass[11pt,reqno]{amsart}

\title{Self-Reachable Chip Configurations on Trees}

\author{Benjamin Lyons, McCabe Olsen}
\address{Department of Mathematics\\
  Rose-Hulman Institute of Technology\\
  Terre Haute, IN 47803--3920}
\date{\today}
\thanks{The first author was a participant at the Rose-Hulman REU funded by NSF DMS \#1852132 and advised by the second author.
The second author was partially supported by NSF DMS \# 2418532.}


\usepackage{amsmath,amssymb, amscd, amsmath, bm, mathtools}
\usepackage{microtype}
\usepackage{mathrsfs}
\usepackage{graphics}
\usepackage{graphicx}
\usepackage{tikz}
\usepackage{pdfpages}
\usepackage{pgfplots}
\usepackage{pgfplotstable}
\usepackage{tkz-fct}
\usepackage{latexsym}
\usepackage{graphicx}
\usepackage{multirow}
\usepackage{cite}
\usepackage{subfig}
\usepackage{enumitem}
\usepackage[bottom]{footmisc}
\usepackage[margin=1.25in]{geometry}
\usetikzlibrary{pgfplots.polar}
\pgfplotsset{compat=newest}
\usepgfplotslibrary{fillbetween}
\usetikzlibrary{patterns}
\input epsf

\numberwithin{equation}{section}
\numberwithin{figure}{section}
\newtheorem{theorem}{Theorem}[section]
\newtheorem{lemma}[theorem]{Lemma}
\newtheorem{corollary}[theorem]{Corollary}
\newtheorem{proposition}[theorem]{Proposition}
\theoremstyle{definition}
\newtheorem{definition}[theorem]{Definition}

\theoremstyle{remark}
\newtheorem{remark}[theorem]{Remark}

\def\XXint#1#2#3{{\setbox0=\hbox{$#1{#2#3}{\int}$ }
\vcenter{\hbox{$#2#3$ }}\kern-.6\wd0}}

\newcommand{\Def}[1]{\textbf{#1}}
\newcommand{\Z}{\mathbb{Z}}

\begin{document}

\begin{abstract} 
In this paper, we explore the notion of a \emph{self-reachable} chip configuration on a simple graph, that is a chip configuration which can be re-obtained from itself after a (nonempty) sequence of vertex firings. 
In particular, we focus on the case of trees and provide a characterization for such configurations, as well as show that all self-reachable configurations with the same number of chips on a tree are reachable from one another. 
We conclude with a recursive enumeration formula for the number of self-reachable configurations.   
\end{abstract}

\maketitle

\section{Introduction}

Let $G$ be a graph.
The chip firing game on $G$ is a rather simple way to explain a discrete dynamical system on the graph: To each vertex, associate a nonnegative integer (often viewed as a commodity such as poker chips). Then choosing a vertex which has at least as many chips as its degree, we ``fire" the vertex. That is, we transfer one chip to each of its neighbors.
We then continue this process as long as we want or are able to. 
The notion of chip firing developed both as a combinatorial game (see, e.g. \cite{ALS89,BLS91,Spencer}) and via the abelian sandpile model, which was first studied by Bak, Tang, and Wisenfeld (see \cite{BakTangWisenfeld}).
The mathematics of chip firing is quite rich encompassing combinatorics, statistical physics, and algebraic geometry via Riemann-Roch theory (see \cite{Klivans} for a broad overview of chip firing). 

In this paper, we ask a rather naive question: given a graph $G$ and an initial configuration chips on its vertices, what must be true in order to return to this configuration after a nonzero number of vertex firings? 
We call such a configuration of chips a \emph{self-reachable configuration}. 
We particularly focus on analyzing these configurations in trees, where one can prove stronger results. 

The paper is structured as follows. 
In Section \ref{sec:background} we outline necessary background, definitions and notation. 
In Section \ref{sec:firingGen}, we state and prove general results of chip firing configurations.
Section \ref{Sec:self} focuses on the main results of the paper on self-reachable configurations, particularly those concerning trees.
Lastly, we conclude briefly in Section \ref{sec:combinatorial} with an enumerative result. 



\section{Background and Definitions} \label{sec:background}

We begin by setting notation which will be used throughout.

\begin{itemize}
	\item For $a, b \in \mathbb{Z}$, we define the sequences $\mathbb{N}_{a} = \{a, a + 1, a + 2, \ldots\}$, and $\mathbb{N}_{[a, b]} = \{a, a + 1, a + 2, \ldots, b\}$.
        \item Given two finite sequences $\Phi$ and $\Psi$, we let $\{\Phi, \Psi\}$ be the sequence obtained by appending the sequence $\Psi$ to the end of the sequence $\Phi$.
        \item For a vector $\bm{x} \in \mathbb{R}^n$, we clarify the dimension of $\bm{x}$ by writing $\bm{x}^{(n)}$.
        \item Given vectors $\bm{x} = (x_1, x_2, \ldots, x_n) \in \mathbb{R}^n$ and $\bm{y} = (y_1, y_2, \ldots, y_m) \in \mathbb{R}^m$, we let \[(\bm{x}, \bm{y}) = (x_1, x_2, \ldots, x_n, y_1, y_2, \ldots, y_m) \in \mathbb{R}^{n + m}.\]
        \item Unless otherwise specified, we will label the vertices of an $n$-vertex graph as $v_1, v_2, \ldots, v_n$. Throughout this paper, we will only work with graphs that have at least one vertex.
        \item We will let $\deg^{(G)}(v_a)$ refer to the degree of the vertex $v_a$ on the graph $G$.
        When the graph is clear from context, we will denote this as simply $\deg(v_a)$. 
\end{itemize}

We will now review some basic definitions from graph theory and chip firing. 
For additional background and detail beyond what is provided here, we refer the reader to the excellent book of Klivans \cite{Klivans}.
Let a simple graph $G$ on $n$ vertices.
Recall that the \Def{Laplacian matrix} of $G$, denoted $\Delta(G)$, is the $n \times n$ matrix given by

     \begin{align*}
         \Delta_{ij}(G) = 
         \begin{cases}
             -1 & i \neq j \text{ and } (v_i, v_j) \in E \\
             \deg(v_i) & i = j \\
             0 & \text{otherwise.}
         \end{cases}
     \end{align*}
A \Def{chip configuration} on $G$ is a vector $\bm{c}^{(n)}\in\Z_{\geq 0}^n$, where we interpret $c_i$ as the number of chips on vertex $v_i$.
If $G$ has $n$ vertices with chip configuration $\bm{c}^{(n)}$, the vector $\bm{d}^{(n)}$ that results from \Def{firing} the $i$th vertex $v_i$ is
	\begin{align*}
        \bm{d}^{(n)} = \bm{c}^{(n)} - \Delta(G)\bm{e}_i^{(n)}.
    \end{align*}
This has the effect of moving a chip from $v_i$ to each vertex neighboring $v_i$. 
We say that it is \Def{legal} to fire $v_i$ on $G$ starting from $\bm{c}^{(n)}$ if $\bm{d}^{(n)}$ is a chip configuration, which is to say the $i$th component of $\bm{d}^{(n)}$ remains nonnegative. 
Equivalently, it is legal to fire $v_i$ provided $\bm{e}_i^{(n)} \cdot \bm{c}^{(n)} \ge \deg(v_i)$.
A \Def{firing sequence} is a sequence of integers which represent a sequence of legal chip firings.
Throughout this paper, firing sequences will be assumed to be finite. 
Given a starting configuration $\bm{c}^{(n)}$ and a firing sequence $\Phi$, the resulting configuration is
	\begin{align}\label{eq2.1}
        \bm{F}_\Phi^{(G)}\left(\bm{c}^{(n)}\right) = \bm{c}^{(n)} - \Delta(G)\sum_{j \in \Phi}\bm{e}_j^{(n)}.
    \end{align}
We say that a firing sequence $\Phi = \{\phi_1, \phi_2, \ldots, \phi_m\}$ is \Def{legal} on $G$ starting from a chip configuration $\bm{c}^{(n)}$ if it consists entirely of legal vertex firings.
That is, it is legal to fire $v_{\phi_1}$ on $G$ starting from $\bm{c}^{(n)}$, it is legal to fire $v_{\phi_2}$ on $G$ starting from $\bm{F}_{\{\phi_1\}}^{(G)}\left(\bm{c}^{(n)}\right)$, it is legal to fire $v_{\phi_3}$ on $G$ starting from $\bm{F}_{\{\phi_1, \phi_2\}}^{(G)}\left(\bm{c}^{(n)}\right)$, and so on.
 Let $R^{(G)}\left(\bm{c}^{(n)}\right)$ to be the set of chip configurations on $G$ that can be reached via nonempty legal firing sequences on $G$ starting from $\bm{c}^{(n)}$.
 We can now introduce a critical definition for this work.
 
 \begin{definition}\label{selfreachableconfig}
    Let $G$ be a connected simple $n$-vertex graph. 
    The chip configuration $\bm{s}^{(n)}$ is called \Def{self-reachable} on $G$ if $\bm{s}^{(n)} \in R^{(G)}\left(\bm{s}^{(n)}\right)$. 
    Let $S_\ell^{(G)}$ denote set of all self-reachable configurations on $G$ with exactly $\ell$ chips.
\end{definition}

\section{Firing Sequences}\label{sec:firingGen}

In this section, we detail
several important background results pertaining to chip firing.
Many of these results are relatively straightforward and implicit from prior work on chip firing. 
Our purpose in stating and proving these results here is to create all necessary tools and techniques for general firing sequences needed to prove the main results in Section \ref{Sec:self}.

 We first recall a well-known result from chip firing theory, that chip firing is an abelian process.
 The proof of this result is well-known and follows immediately from the definitions.

\begin{lemma}\label{abelianprocess}
    Let $G$ be a simple $n$-vertex graph with chip configuration $\bm{c}^{(n)}$. Let $\Phi$ and $\Psi$ be legal firing sequences on $G$ starting from $\bm{c}^{(n)}$ with the property that $\Phi$ and $\Psi$ contain each natural number between 1 and $n$ the same number of times (i. e. they fire each vertex of $G$ the same number of times). Then $\bm{F}_\Phi^{(G)}\left(\bm{c}^{(n)}\right) = \bm{F}_\Psi^{(G)}\left(\bm{c}^{(n)}\right)$.
\end{lemma}




The following proposition describes an additive property associates with legal firing sequences.
It will be useful in subsequent results. 

\begin{proposition}\label{AdditivityOfFiring}
    Let $G$ be a simple $n$-vertex graph. Then for any chip configuration $\bm{c}^{(n)}$ on $G$, any vector $\bm{x}^{(n)} \in \mathbb{Z}^n$ such that $\bm{c}^{(n)} + \bm{x}^{(n)}$ is still a valid chip configuration, and any firing sequence $\Phi$ on $G$,
    \begin{align*}
        \bm{F}_\Phi^{(G)}\left(\bm{c}^{(n)} + \bm{x}^{(n)}\right) = \bm{F}_\Phi^{(G)}\left(\bm{c}^{(n)}\right) + \bm{x}^{(n)}.
    \end{align*}
\end{proposition}

\begin{proof}
    Following from equation \eqref{eq2.1}, we have
    \begin{align*}
        \bm{F}_\Phi^{(G)}\left(\bm{c}^{(n)} + \bm{x}^{(n)}\right) = \bm{c}^{(n)} + \bm{x}^{(n)} - \Delta(G)\sum_{j \in \Phi}\bm{e}_j^{(n)}
    \end{align*}
    and
    \begin{align*}
        \bm{F}_\Phi^{(G)}\left(\bm{c}^{(n)}\right) + \bm{x}^{(n)} = \bm{c}^{(n)} - \Delta(G)\sum_{j \in \Phi}\bm{e}_j^{(n)} + \bm{x}^{(n)}
    \end{align*}
    which gives the desired result. 
\end{proof}

The result to follow concerns legality of firing sequences on different starting chip configurations. 

\begin{proposition}\label{LegalityProp}
    Let $G$ be a simple $n$-vertex graph with chip configuration $\bm{c}^{(n)}$, and let $\Phi$ be a legal firing sequence on $G$ starting from $\bm{c}^{(n)}$. Let $\bm{d}^{(n)}$ be a chip configuration on $G$ with the property that for each $i \in \mathbb{N}_{[1, n]}$ that appears in $\Phi$,
    \begin{align*}
        \left(\bm{d}^{(n)} - \bm{c}^{(n)}\right) \cdot \bm{e}_i^{(n)} \ge 0.
    \end{align*}
    Then $\Phi$ is a legal firing sequence on $G$ starting from $\bm{d}^{(n)}$.
\end{proposition}

\begin{proof}
    We let $\Phi = \{\phi_1, \phi_2, \ldots, \phi_k\}$, and we let $j \in \mathbb{N}_{[1, k]}$ be arbitrary. It suffices to prove that it is legal to fire $v_{\phi_j}$ starting from $\bm{F}_{\{\phi_1, \phi_2, \ldots, \phi_{j - 1}\}}^{(G)}\left(\bm{d}^{(n)}\right)$. By Proposition \ref{AdditivityOfFiring}, we have that
    \begin{align*}
        \bm{F}_{\{\phi_1, \phi_2, \ldots, \phi_{j - 1}\}}^{(G)}\left(\bm{d}^{(n)}\right) = \bm{F}_{\{\phi_1, \phi_2, \ldots, \phi_{j - 1}\}}^{(G)}\left(\bm{c}^{(n)}\right) + \bm{d}^{(n)} - \bm{c}_1^{(n)}.
    \end{align*}
    Taking the dot product with $\bm{e}_{\phi_j}^{(n)}$, we obtain
    \begin{align}\label{3.3}
        \bm{F}_{\{\phi_1, \phi_2, \ldots, \phi_{j - 1}\}}^{(G)}\left(\bm{d}^{(n)}\right) \cdot \bm{e}_{\phi_j}^{(n)} = \left[\bm{F}_{\{\phi_1, \phi_2, \ldots, \phi_{j - 1}\}}^{(G)}\left(\bm{c}^{(n)}\right) \cdot \bm{e}_{\phi_j}^{(n)}\right] + \left[\left(\bm{d}^{(n)} - \bm{c}^{(n)}\right) \cdot \bm{e}_{\phi_j}^{(n)}\right].
    \end{align}
    Note that the first term in square brackets in \eqref{3.3} is nonnegative by the assumption that $\Phi$ is a legal firing sequence on $G$ starting from $\bm{c}^{(n)}$. The second term in square brackets in \eqref{3.3} is also nonnegative by our assumption that $\left(\bm{d}^{(n)} - \bm{c}^{(n)}\right) \cdot \bm{e}_i^{(n)} \ge 0$ for each $i \in \mathbb{N}_{[1, n]}$ that appears in $\Phi$. 
    Thus, the left-hand side of \eqref{3.3} which yeilds the desired result.
\end{proof}

The following lemma extends the abelian notion given by Lemma \ref{abelianprocess} to disjoint firing sequences. 

\begin{lemma}\label{AbelianProcessII}
    Let $G$ be a simple $n$-vertex graph with chip configuration $\bm{c}^{(n)}$. Let $\Phi$ and $\Psi$ be legal firing sequences on $G$ starting from $\bm{c}^{(n)}$ with the property that $\Phi \cap \Psi = \varnothing$ when $\Phi$ and $\Psi$ are viewed as multisets. Then both $\{\Phi, \Psi\}$ and $\{\Psi, \Phi\}$ are legal firing sequences on $G$ starting from $\bm{c}^{(n)}$.
\end{lemma}

\begin{proof}
    We will prove that $\{\Phi, \Psi\}$ is a legal firing sequence on $G$ starting from $\bm{c}^{(n)}$; the argument for $\{\Psi, \Phi\}$ is identical by symmetry. Note that because none of the elements of $\Psi$ appear in $\Phi$ and vertices can only lose chips by firing, the vertices that fire in $\Psi$ do not lose any chips when the vertices in $\Phi$ are firing. Therefore
    \begin{align*}
        \left(\bm{F}_\Phi^{(G)}\left(\bm{c}^{(n)}\right) - \bm{c}^{(n)}\right) \cdot \bm{e}_i^{(n)} \ge 0
    \end{align*}
    for all $i \in \mathbb{N}_{[1, n]}$ that appear in $\Psi$. So by Proposition \ref{LegalityProp}, $\Psi$ is a legal firing sequence on $G$ starting from $\bm{F}_\Phi^{(G)}\left(\bm{c}^{(n)}\right)$. Since we are given that $\Phi$ is a legal firing sequence on $G$ starting from $\bm{c}^{(n)}$, we conclude that $\{\Phi, \Psi\}$ is a legal firing sequence on $G$ starting from $\bm{c}^{(n)}$, as desired.
\end{proof}




We next prove that any legal firing sequence on a graph that fires each vertex exactly once yields the starting chip configuration.

\begin{lemma}\label{firingsequenceeachvertexonce}
    Let $G$ be a simple $n$-vertex graph with chip configuration $\bm{c}^{(n)}$, and let $\Phi$ be a legal firing sequence on $G$ starting from $\bm{c}^{(n)}$ that contains each natural number between 1 and $n$ exactly once. Then $\bm{F}_\Phi^{(G)}\left(\bm{c}^{(n)}\right) = \bm{c}^{(n)}$.
\end{lemma}

\begin{proof}
    By definition,
    \begin{align*}
        \bm{F}_\Phi^{(G)}\left(\bm{c}^{(n)}\right) = \bm{c}^{(n)} - \Delta(G)\sum_{j \in \Phi}\bm{e}_j^{(n)}.
    \end{align*}
    Now note that since $\Phi$ fires each vertex of $G$ exactly once,
    \begin{align*}
        \sum_{j \in \Phi}\bm{e}_j^{(n)} = \bm{1}^{(n)}.
    \end{align*}
    Since $\bm{1}^{(n)}$ is in the kernel of $\Delta(G)$, 
    \begin{align*}
        \bm{F}_\Phi^{(G)}\left(\bm{c}^{(n)}\right) = \bm{c}^{(n)} - \Delta(G)\bm{1}^{(n)} = \bm{c}^{(n)},
    \end{align*}
    which proves the desired result.
\end{proof}

Our next result is a lemma that establishes a method of developing firing sequences of shorter lengths.

\begin{lemma}\label{firingsequencealgorithm}
    Let $G$ be a simple $n$-vertex graph with chip configuration $\bm{c}^{(n)}$, and let $\Phi$ be a legal firing sequence on $G$ starting from $\bm{c}^{(n)}$ that contains each natural number between 1 and $n$ at least once. Let $\Psi$ be the sequence obtained by removing the first occurrence of each $k \in \mathbb{N}_{[1, n]}$ from $\Phi$. Then:
    
    \begin{enumerate}
        \item $\Psi$ is a legal firing sequence on $G$ starting from $\bm{c}^{(n)}$, and
        \item $\bm{F}_\Phi^{(G)}\left(\bm{c}^{(n)}\right) = \bm{F}_\Psi^{(G)}\left(\bm{c}^{(n)}\right)$.
    \end{enumerate}
\end{lemma}

\begin{proof}
To prove the first claim, suppose towards contradiction that $\Psi$ is an illegal firing sequence on $G$ starting from $\bm{c}^{(n)}$. This means that at least one of the firing moves in $\Psi$ is illegal. Without loss of generality, we can relabel the vertices of $G$ such that for some $m \in \mathbb{N}$, the $m$th firing of $v_1$ in $\Psi$ is illegal. This means that 1 appears at least $m$ times in $\Psi$ and at least $m + 1$ times in $\Phi$. We can then relabel the remaining vertices of $G$ such that for some $j \in \mathbb{N}_{[1, n]}$, $1, 2, \ldots, j$ appear at least once in $\Phi$ before the $(m + 1)$th occurrence of 1, while $j + 1, j + 2, \ldots, n$ do not. We now let $\Phi^*$ be the sequence of elements of $\Phi$ up to but excluding the $(m + 1)$th occurrence of 1, and we let $\Psi^*$ be the sequence of elements of $\Psi$ up to but excluding the $m$th occurrence of 1. So for each $k \in \mathbb{N}_{[1, j]}$, $\Phi^*$ contains exactly one more copy of $k$ than $\Psi^*$. This means that the sequences $\Phi^*$ and $\left\{\Psi^*, \mathbb{N}_{[1, j]}\right\}$ contain each natural number between 1 and $n$ the same number of times. So Lemma \ref{abelianprocess} implies that
        \begin{align*}
            \bm{F}_{\Phi^*}^{(G)}\left(\bm{c}^{(n)}\right) = \bm{F}_{\left\{\Psi^*, \mathbb{N}_{[1, j]}\right\}}^{(G)}\left(\bm{c}^{(n)}\right)
        \end{align*}
        We now let $d = \deg(v_1)$.
        Since the next firing move in $\Psi$ after the firing moves in $\Psi^*$ is the illegal firing of $v_1$, there must be fewer than $d$ chips on $v_1$ in the chip configuration $\bm{F}_{\Psi^*}^{(G)}\left(\bm{c}^{(n)}\right)$. This is still true for $\bm{F}_{\left\{\Psi^*, \mathbb{N}_{[1, j]}\right\}}^{(G)}\left(\bm{c}^{(n)}\right)$ since $v_1$ fires an additional time and each neighbor of $v_1$ fires at most one additional time. Therefore, it is illegal to fire $v_1$ on $G$ starting from $\bm{F}_{\left\{\Psi^*, \mathbb{N}_{[1, j]}\right\}}^{(G)}\left(\bm{c}^{(n)}\right) = \bm{F}_{\Phi^*}^{(G)}\left(\bm{c}^{(n)}\right)$. This contradicts the fact that $\Phi$ is a legal firing sequence since 1 is the first element of $\Phi$ that is not part of $\Phi^*$.

To see the second claim, note that 
        \begin{align*}
            \bm{F}_{\Phi}^{(G)}\left(\bm{c}^{(n)}\right) = \bm{F}_{\left\{\Psi, \mathbb{N}_{[1, n]}\right\}}^{(G)}\left(\bm{c}^{(n)}\right) = \bm{F}_{\mathbb{N}_{[1, n]}}^{(G)}\left(\bm{F}_{\Psi}^{(G)}\left(\bm{c}^{(n)}\right)\right)
        \end{align*}
     by Lemma \ref{abelianprocess} as $\Phi$ and $\left\{\Psi, \mathbb{N}_{[1, n]}\right\}$ contain each natural number the same number of times. By Lemma \ref{firingsequenceeachvertexonce},
        \begin{align*}
            \bm{F}_{\mathbb{N}_{[1, n]}}^{(G)}\left(\bm{F}_{\Psi}^{(G)}\left(\bm{c}^{(n)}\right)\right) = \bm{F}_{\Psi}^{(G)}\left(\bm{c}^{(n)}\right).
        \end{align*}
        Thus,
        \begin{align*}
            \bm{F}_{\Phi}^{(G)}\left(\bm{c}^{(n)}\right) = \bm{F}_{\Psi}^{(G)}\left(\bm{c}^{(n)}\right),
        \end{align*}
        as desired.
\end{proof}

Next, we present a lemma that concerns firing a vertex of a graph that is attached to a vertex of degree 1.

\begin{lemma}\label{treeinductionbasis}
    Let $G$ be a simple $n$-vertex graph ($n \in \mathbb{N}_2$) where $v_n$ is a vertex of degree 1 connected to $v_{n - 1}$. Suppose $G$ has initial chip configuration $\left(\bm{c}^{(n - 2)}, c_{n - 1}, c_n\right)$, where $c_n \ge 1$. If it is legal to fire $v_{n - 1}$ on $G \setminus v_n$ starting from $\left(\bm{c}^{(n - 2)}, c_{n - 1}\right)$, then:
    \begin{enumerate}
        \item The firing sequence $\{n, n - 1\}$ is legal on $G$ starting from $\left(\bm{c}^{(n - 2)}, c_{n - 1}, c_n\right)$, and
        \item
        \begin{align}\label{eq3.3}
            \bm{F}_{\{n, n - 1\}}^{(G)}\left(\left(\bm{c}^{(n - 2)}, c_{n - 1}, c_n\right)\right) = \left(\bm{F}_{\{n - 1\}}^{(G \setminus v_n)}\left(\left(\bm{c}^{(n - 2)}, c_{n - 1}\right)\right), c_n\right)
        \end{align}
    \end{enumerate}
\end{lemma}

\begin{proof}
      For the first claim,  since $c_n \ge 1$, the firing of $v_n$ on $G$ starting from $\left(\bm{c}^{(n - 2)}, c_{n - 1}, c_n\right)$ is legal, so we only need to show that the firing of $v_{n - 1}$ on $G$ starting from $\left(\bm{c}^{(n - 2)}, c_{n - 1} + 1, c_n - 1\right)$ is legal. First note that
        \begin{align*}
            \deg^{(G)}(v_{n - 1}) = \deg^{(G \setminus v_n)}(v_{n - 1}) + 1
        \end{align*}
        since $v_{n - 1}$ and $v_n$ share an edge. Because it is legal to fire $v_{n - 1}$ on $G \setminus v_n$ starting from $\left(\bm{c}^{(n - 2)}, c_{n - 1}\right)$, we have that $c_{n - 1} \ge \deg^{(G \setminus v_n)}(v_{n - 1})$. Thus, 
        \begin{align*}
            c_{n - 1} + 1 \ge \deg^{(G \setminus v_n)}(v_{n - 1}) + 1 = \deg^{(G)}(v_{n - 1}).
        \end{align*}
        Thus, it is legal to fire $v_{n - 1}$ on $G$ starting from $\left(\bm{c}^{(n - 2)}, c_{n - 1} + 1, c_n - 1\right)$, which proves (1).
        
        To show the second claim, note that first $n - 2$ components of the chip configurations in \eqref{eq3.3} are the same because $v_{n - 1}$ fires once in both firing sequences and these components are not affected by the presence or absence of $v_n$. The second-to-last component is also the same in both cases because the firing of $v_n$ on $G$ replenishes the chip that $v_{n - 1}$ sends to $v_n$ on $G$. Finally, the last component of the left-hand side is $c_n$ since $v_n$ loses a chip when it fires on $G$ and regains it when $v_{n - 1}$ fires. This completes the proof.

\end{proof}

\section{Self-Reachable Configurations}\label{Sec:self}



In this section, our focus lies with self-reachable configurations. 
In particular, we will prove the following results on self-reachable configurations on trees.

\begin{theorem}\label{srcsubtreeresult}
    Let $T$ be an $n$-vertex tree. Then the configuration $\bm{s}^{(n)}$ is self-reachable on $T$ if and only if $\bm{s}^{(n)}$ has at least $m - 1$ chips on every $m$-vertex subtree of $T$.
\end{theorem}

\begin{theorem}\label{srcmutualreachability}
    Let $T$ be an $n$-vertex tree, and let $\bm{s}^{(n)}$ be a self-reachable chip configuration on $T$. Then for any chip configuration $\bm{c}^{(n)}$ with the same number of chips as $\bm{s}^{(n)}$, $\bm{c}^{(n)} \in R^{(T)}\left(\bm{s}^{(n)}\right)$ if and only if $\bm{c}^{(n)}$ is also self-reachable.
\end{theorem}

The significance of this second result in particular is that, in the case of trees, all self-reachable configurations lie in the same orbit. 
Therefore, one can simultaneously consider all self-reachable configurations on a tree with the same number of chips. 
In order to show these results, we need to build up some machinery  regarding self-reachable configurations, not only on trees, but on connected simple graphs in general. 
We begin with the following existence result on legal firing sequences for self-reachable configurations.

\begin{proposition}\label{srcfireeachvertexonce}
    Let $G$ be a connected simple $n$-vertex graph, and let $\bm{s}^{(n)}$ be a self-reachable configuration on $G$. Let $v_i$ be a vertex that can legally fire on $G$ starting from $\bm{s}^{(n)}$. Then there exists a legal firing sequence $\Phi = \{\phi_1, \phi_2, \ldots, \phi_n\}$ on $G$ starting from $\bm{s}^{(n)}$ such that:

    \begin{enumerate}
        \item $\Phi$ contains each natural number between 1 and $n$ exactly once, and

        \item $\phi_1 = i$.
    \end{enumerate}
\end{proposition}

\begin{proof}
First we show that there exists a legal firing sequence $\Psi$ on $G$ starting from $\bm{s}^{(n)}$ that satisfies (1).
By Definition \ref{selfreachableconfig}, for some legal $n$-vertex firing sequence $\Pi$,
        \begin{equation}
            \bm{F}_\Pi^{(G)}\left(\bm{s}^{(n)}\right) = \bm{s}^{(n)}.\notag
        \end{equation}
        Suppose towards contradiction that $\Pi$ does not contain each integer between 1 and $n$ the same number of times. Then by Lemma \ref{firingsequencealgorithm}, we can assume without loss of generality that for some $t \in \mathbb{N}_{[1, n]}$, $t \notin \Pi$, meaning $v_t$ never fires, and also that $\Pi \neq \varnothing$. Since $v_t$ never fires, we know that none of the neighbors of $v_t$ can fire or else $v_t$ will end up with more chips than it started with. By the same logic, the neighbors of these neighbors cannot fire either. By iterating this argument, we deduce that no vertex on $G$ can ever fire because $G$ is connected, so $\Pi = \varnothing$, which is a contradiction. Thus, $\Pi$ contains each integer between 1 and $n$ the same number of times, and thus by Lemma \ref{firingsequencealgorithm}, we can reduce $\Pi$ to a legal $n$-vertex firing sequence that fires each vertex exactly once by removing all but the last occurrence of each element of $\mathbb{N}_{[1, n]}$ from $\Pi$. We let $\Psi$ be the resulting firing sequence, which satisfies (1).
        
        We now show that we can reorder the elements of $\Psi$ to form a new firing sequence that satisfies (2) while preserving legality.
        Without loss of generality, we can relabel the vertices of $G$ such that $\Psi = \mathbb{N}_{[1, n]}$, and we assume that $v_i$ is chosen under this new vertex labeling. Then $\Phi = \{i, 1, 2, \ldots, i - 1, i + 1, \ldots, n\}$ satisfies both (1) and (2), so we only need to check that $\Phi$ is a legal firing sequence on $G$ starting from $\bm{s}^{(n)}$. Because $\Psi = \mathbb{N}_{[1, n]}$ is a legal firing sequence on $G$ starting from $\bm{s}^{(n)}$, the sequence $\mathbb{N}_{[1, i - 1]}$ is also a legal firing sequence on $G$ starting from $\bm{s}^{(n)}$. Since the firing sequence $\{i\}$ is legal on $G$ starting from $\bm{s}^{(n)}$, Lemma \ref{AbelianProcessII} guarantees that $\left\{\{i\}, \mathbb{N}_{[1, i - 1]}\right\} = \{i, 1, 2, \ldots, i - 1\}$ is a legal firing sequence on $G$ starting from $\bm{s}^{(n)}$. Also note that
        \begin{align*}
            \bm{F}_{\left\{\{i\}, \mathbb{N}_{[1, i - 1]}\right\}}^{(G)}\left(\bm{s}^{(n)}\right) = \bm{F}_{\mathbb{N}_{[1, i]}}^{(G)}\left(\bm{s}^{(n)}\right)
        \end{align*}
        by Lemma \ref{abelianprocess}. Since $\Psi$ is a legal firing sequence on $G$ starting from $\bm{s}^{(n)}$, we have that $\mathbb{N}_{[i + 1, n]}$ is a legal firing sequence on $G$ starting from $\bm{F}_{\mathbb{N}_{[1, i]}}^{(G)}\left(\bm{s}^{(n)}\right)$. This guarantees that $\mathbb{N}_{[i + 1, n]}$ is a legal firing sequence on $G$ starting from $\bm{F}_{\left\{\{i\}, \mathbb{N}_{[1, i - 1]}\right\}}^{(G)}\left(\bm{s}^{(n)}\right)$ as well, meaning $\left\{\{i\}, \mathbb{N}_{[1, i - 1]}, \mathbb{N}_{i + 1}^{n}\right\}$ is a legal firing sequence on $G$ starting from $\bm{s}^{(n)}$. But $\left\{\{i\}, \mathbb{N}_{[1, i - 1]}, \mathbb{N}_{i + 1}^{n}\right\} = \Phi$, so the proof is complete.
        
\end{proof}

The following technical lemma shows that a legal firing of a vertex in a self-reachable configuration yields another self-reachable configuration.
More significantly, this allows us to conclude that only self-reachable configurations can be obtained from a self-reachable configuration and further that these self-reachable configurations are mutually reachable from one another.

\begin{lemma}\label{closureofsrcs}
   Let $G$ be a connected simple $n$-vertex graph, and let $\bm{s}^{(n)}$ be a self-reachable configuration on $G$. Suppose it is legal to fire $v_i$ on $G$ starting from $\bm{s}^{(n)}$. Then:

    \begin{enumerate}
        \item $\bm{s}^{(n)} \in R^{(G)}\left(\bm{F}_{\{i\}}^{(G)}\left(\bm{s}^{(n)}\right)\right)$, and
        
        \item $\bm{F}_{\{i\}}^{(G)}\left(\bm{s}^{(n)}\right)$ is self-reachable on $G$.
        
        \item For all $\bm{c}^{(n)} \in R^{(G)}\left(\bm{s}^{(n)}\right)$, $\bm{c}^{(n)}$ is self-reachable.
        
        \item For all $\bm{c}^{(n)} \in R^{(G)}\left(\bm{s}^{(n)}\right)$, $\bm{s}^{(n)} \in R^{(G)}\left(\bm{c}^{(n)}\right)$
    \end{enumerate}

\end{lemma}

\begin{proof}
    The case of $n=1$ is trivial, so let $n\geq 2$.
    To show (1), let $\bm{d}^{(n)} = \bm{F}_{\{i\}}^{(G)}\left(\bm{s}^{(n)}\right)$. We know by Proposition \ref{srcfireeachvertexonce} that there exists a legal firing sequence $\Phi = \{\phi_1, \phi_2, \ldots, \phi_n\}$ that fires each vertex of $G$ exactly once with
        \begin{equation}
            \bm{F}_{\Phi}^{(G)}\left(\bm{s}^{(n)}\right) = \bm{s}^{(n)}.\notag
        \end{equation}
        and $\phi_1 = i$. Then if we let $\Psi = \{\phi_2, \phi_3, \ldots, \phi_n\}$, we have that
        \begin{equation}
            \bm{s}^{(n)} = \bm{F}_{\Psi}^{(G)}\left(\bm{d}^{(n)}\right).\notag
        \end{equation}

        Since $\Psi$ is nonempty for $n \ge 2$, this proves (1).
        
        To see (2), observe that
        \begin{align*}
            \bm{F}_{\{\Psi, \{i\}\}}^{(G)}\left(\bm{d}^{(n)}\right) = \bm{F}_{\{i\}}^{(G)}\left(\bm{s}^{(n)}\right) = \bm{d}^{(n)},
        \end{align*}
        which implies $\bm{d}^{(n)}$ is self-reachable on $G$. 
        Moreover, note that (3) follows immediately by repeated application of this observation.
        
       To see (4), let $\Phi = \{\phi_i\}_{i = 1}^k$ be a legal firing sequence on $G$ such that $\bm{c}^{(n)} = \bm{F}_{\Phi}\left(\bm{s}^{(n)}\right)$. 
       By (1), we know that for each $j \in \mathbb{N}_{[1, k]}$,
    \begin{equation}
        \bm{F}_{\{\phi_i\}_{i = 1}^{j - 1}}\left(\bm{s}^{(n)}\right) \in R^{(G)}\left(\bm{F}_{\{\phi_i\}_{i = 1}^{j}}\left(\bm{s}^{(n)}\right)\right).\notag
    \end{equation}
    Therefore, we can reach $\bm{F}_{\{\phi_i\}_{i = 1}^{k - 1}}\left(\bm{s}^{(n)}\right)$ from $\bm{c}^{(n)}$, $\bm{F}_{\{\phi_i\}_{i = 1}^{k - 2}}\left(\bm{s}^{(n)}\right)$ from $\bm{F}_{\{\phi_i\}_{i = 1}^{k - 1}}\left(\bm{s}^{(n)}\right)$, and so on, until we can reach $\bm{s}^{(n)}$ from $\bm{F}_{\{\phi_1\}}\left(\bm{s}^{(n)}\right)$. This completes the proof.




%
   

\end{proof}

As an aside, the results of Lemma \ref{closureofsrcs} immediately yield that reachability forms an equivalence relation on the set of self-reachable configurations. 

\begin{corollary}\label{srcequivalencerelation}
    Let $G$ be a connected simple $n$-vertex graph. Define a relation $\sim$ on $S_\ell^{(G)}$ as follows: $\bm{s}_1^{(n)} \sim \bm{s}_2^{(n)}$ if $\bm{s}_1^{(n)} \in R^{(G)}\left(\bm{s}_2^{(n)}\right)$. Then $\sim$ is an equivalence relation on $S_\ell^{(G)}$.
\end{corollary}


We now have the necessary tools to prove the first main result of this section, Theorem \ref{srcsubtreeresult}.

\begin{proof}[Proof of Theorem \ref{srcsubtreeresult}]
    To prove sufficiency, we assume that $\bm{c}^{(n)}$ is such that for some subtree $T^*$ of $T$ with $m$ vertices, $\bm{c}^{(n)}$ has fewer than $m - 1$ chips on $T^*$. We choose $T^*$ to have the smallest possible number of vertices while still satisfying this property. Without loss of generality, suppose the vertices of $T^*$ are $v_1, v_2, \ldots, v_m$. Note that if $m = 1$, then $\bm{c}^{(n)}$ must have fewer than 0 chips on $v_1$, which is impossible because we are assuming $\bm{c}^{(n)}$ is a valid chip configuration. Therefore, we can assume $m \ge 2$. Suppose towards contradiction that on some vertex $v_i$ of $T^*$, $\bm{c}^{(n)}$ has at least $d = \deg^{(T^*)}(v_i)$ chips. Note that the graph $T^* \setminus v_i$ contains $d$ disjoint trees. Let $T_1^*, T_2^*, \ldots, T_d^*$ be the disjoint trees in $T^* \setminus v_i$, and let $m_j^*$ be the number of vertices on $T_j^*$ for each $j \in \mathbb{N}_{[1, d]}$. Because we chose $T^*$ to have the smallest possible number of vertices, we know that for each $j \in \mathbb{N}_{[1, d]}$, $\bm{c}^{(n)}$ has at least $m_j^* - 1$ chips on $T_j^*$. This makes for a total of at least
    \begin{equation}
        d + \sum_{j = 1}^{d}(m_j^* - 1) = d - d + \sum_{j = 1}^{d}m_j^* = m - 1\notag
    \end{equation}
    chips on $T^*$. This contradicts our assumption that $\bm{c}^{(n)}$ has fewer than $m - 1$ chips on $T^*$. Therefore, for each $i \in \mathbb{N}_{[1, m]}$, $\bm{c}^{(n)}$ has fewer than $\deg^{(T^*)}(v_i)$ chips on $v_i$. Now suppose towards contradiction that $\bm{c}^{(n)}$ is self-reachable on $T$. Then by Proposition \ref{srcfireeachvertexonce}, there is a legal $n$-vertex firing sequence $\Phi$ that fires each vertex of $T$ exactly once such that
    \begin{equation}
        \bm{F}_{\Phi}^{(T)}\left(\bm{c}^{(n)}\right) = \bm{c}^{(n)}.\notag
    \end{equation}
    Now let $a \in \mathbb{N}_{[1, m]}$ be the index of the first vertex in $T^*$ that fires in $\Phi$, and let $N = \deg^{(T)}(v_a) - \deg^{(T^*)}(v_a)$. Note that $N$ counts the number of vertices that neighbor $v_a$ in $T$ that are not in $T^*$. Based on what we just showed, $\bm{c}^{(n)}$ has at most $\deg^{(T^*)}(v_a) - 1 = \deg^{(T)}(v_a) - N - 1$ chips on $v_a$. Because $v_a$ is the first vertex to fire in $T^*$, at most $N$ neighbors of $v_a$ can fire in $\Phi$ before $v_a$ must fire itself. Therefore, $v_a$ can have at most $\deg^{(T)}(v_a) - 1$ chips right before it fires. This implies that the firing of $v_a$ is not legal, contradicting the legality of $\Phi$. Therefore, $\bm{c}^{(n)}$ cannot be self-reachable on $T$, which proves sufficiency.

    To prove necessity, we induct on the number of vertices in $T$. On a tree with one vertex, any configuration with 0 or more chips is trivially self-reachable because firing the single vertex is legal (the vertex has degree 0) and produces the same chip configuration. This completes our base case. For our induction hypothesis, we suppose that on any tree $T$ with $k$ vertices, a chip configuration $\bm{c}^{(k)}$ is self-reachable on $T$ only if $\bm{c}^{(k)}$ has at least $m - 1$ chips on every $m$-vertex subtree of $T$. We now suppose $T$ is a tree with $k + 1$ vertices. Without loss of generality, we can assume $v_{k + 1}$ is a leaf, and we let $v_t$ be the vertex with which it shares an edge. We let $\left(\bm{c}^{(k)}, c_{k + 1}\right)$ be a chip configuration on $T$ with the property that on every $m$-vertex subtree of $T$, $\left(\bm{c}^{(k)}, c_{k + 1}\right)$ has at least $m - 1$ chips. To complete the proof, we must consider the following two cases:
    \begin{enumerate}
    \item $c_{k + 1} > 0$
    \item $c_{k+1}=0$
    \end{enumerate}
    To handle (1), note that because all subtrees of $T \setminus v_{k + 1}$ are also subtrees of $T$, we know that $\bm{c}^{(k)}$ has at least $m - 1$ chips on every $m$-vertex subtree of $T \setminus v_{k + 1}$ and is therefore self-reachable on $T \setminus v_{k + 1}$ by our induction hypothesis. We let $\Phi$ be a legal firing sequence on $T \setminus v_{k + 1}$ that fires each vertex in $T \setminus v_{k + 1}$ exactly once such that
        \begin{equation}
            \bm{F}_{\Phi}^{(T \setminus v_{k + 1})}\left(\bm{c}^{(k)}\right) = \bm{c}^{(k)}.\notag
        \end{equation}
        Without loss of generality, we can label the vertices of $T \setminus v_{k + 1}$ such that $\Phi = \mathbb{N}_{[1, k]}$. We create a $k + 1$-vertex firing sequence $\Psi$ by adding $k + 1$ before $t$ in $\Phi$, so $\Psi = \{1, 2, \ldots, t - 1, k + 1, t, t + 1, \ldots, k\}$. Note by Lemma \ref{firingsequenceeachvertexonce} that
        \begin{equation}
            \bm{F}_{\Psi}^{(T)}\left(\left(\bm{c}^{(k)}, c_{k + 1}\right)\right) = \left(\bm{c}^{(k)}, c_{k + 1}\right).\notag
        \end{equation}
        Therefore, all that is left to prove is the legality of $\Psi$. Because $v_1, v_2, \ldots, v_{t - 1}$ do not neighbor $v_{k + 1}$, the legality of the first $t - 1$ firing moves in $\Psi$ follows from the fact that $\Phi$ is a legal firing sequence on $G$ starting from $\bm{c}^{(n)}$. The legality of $\Phi$ also implies that it is legal to fire $v_t$ on $T \setminus v_{k + 1}$ starting from $\bm{F}_{\mathbb{N}_{[1, t - 1]}}^{(T \setminus v_{k + 1})}\left(\bm{c}^{(k)}\right)$. So by Lemma \ref{treeinductionbasis}, $\{k + 1, t\}$ is a legal firing sequence on $T$ starting from $\bm{F}_{\mathbb{N}_{[1, t - 1]}}^{(T \setminus v_{k + 1})}\left(\bm{c}^{(k)}\right)$, and furthermore,
        \begin{align*}
            \bm{F}_{\left\{\mathbb{N}_{[1, t - 1]}, \{k + 1, t\}\right\}}^{(T)}\left(\left(\bm{c}^{(k)}, c_{k + 1}\right)\right) = \left(\bm{F}_{\mathbb{N}_{[1, t]}}^{(T \setminus v_{k + 1})}\left(\bm{c}^{(k)}\right), c_{k + 1}\right).
        \end{align*}
        Due to this equality, the legality of the remaining firing moves in $\Psi$ (which all involve vertices that are not and do not neighbor $v_{k + 1}$) follows from the legality of $\Phi$. Therefore, $\left(\bm{c}^{(k)}, c_{k + 1}\right)$ is self-reachable on $T$ in this case.

For (2), suppose towards contradiction that on some subtree $T^*$ of $T \setminus v_{k + 1}$ with $m$ vertices, the chip configuration $\bm{c}^{(k)} - \bm{e}_t^{(k)}$ has fewer than $m - 1$ chips. Because $\left(\bm{c}^{(k)}, c_{k + 1}\right)$ must have at least $m - 1$ chips on $T^*$ by assumption, this is only possible if $v_t$ is a vertex of $T^*$ and $\bm{c}^{(k)} - \bm{e}_t^{(k)}$ has exactly $m - 2$ chips on $T^*$. Then if we let $T^{**}$ be the subtree of $T$ formed by joining $v_{k + 1}$ to $v_t$ on $T^*$, we see that $\left(\bm{c}^{(k)}, c_{k + 1}\right)$ has $m - 1$ chips on $T^{**}$, an $(m + 1)$-vertex subtree, which is a contradiction. This means that $\bm{c}^{(k)} - \bm{e}_t^{(k)}$ has at least $m - 1$ chips on all $m$-vertex subtrees of $T \setminus v_{k + 1}$. Therefore, $\bm{c}^{(k)} - \bm{e}_t^{(k)}$ is self-reachable on $T \setminus v_{k + 1}$ by our induction hypothesis. We let $\Phi$ be a legal $k$-vertex firing sequence that fires each vertex in $T \setminus v_{k + 1}$ exactly once such that
        \begin{equation}
            \bm{F}_{\Phi}^{(T \setminus v_{k + 1})}\left(\bm{c}^{(k)} - \bm{e}_t^{(k)}\right) = \bm{c}^{(k)} - \bm{e}_t^{(k)}.\notag
        \end{equation}
        Without loss of generality, let $\Phi = \mathbb{N}_{[1, k]}$. We create a $k + 1$-vertex firing sequence $\Psi$ by adding $k + 1$ to the end of $\Phi$, so $\Psi = \mathbb{N}_{[1, k + 1]}$. By Lemma \ref{firingsequenceeachvertexonce},
        \begin{equation}
            \bm{F}_{\Psi}^{(T)}\left(\left(\bm{c}^{(k)}, c_{k + 1}\right)\right) = \left(\bm{c}^{(k)}, c_{k + 1}\right).\notag
        \end{equation}
        The last step is to prove that $\Psi$ is a legal firing sequence on $T$ starting from $\left(\bm{c}^{(k)}, c_{k + 1}\right)$. The legality of the first $t - 1$ firing moves in $\Psi$ follows from the legality of $\Phi$ (since none of these firing moves are affected by the presence of $v_{k + 1}$ or the number of chips on $v_t$). Now note that because $\Phi$ is a legal firing sequence on $T \setminus v_{k + 1}$ starting from $\bm{c}^{(k)} - \bm{e}_t^{(k)}$, it is legal to fire $v_t$ on $T \setminus v_{k + 1}$ starting from $\bm{F}_{\mathbb{N}_{[1, t - 1]}}^{(T \setminus v_{k + 1})}\left(\bm{c}^{(k)} - \bm{e}_t^{(k)}\right)$. This means that $\bm{e}_t^{(k)} \cdot \bm{F}_{\mathbb{N}_{[1, t - 1]}}^{(T \setminus v_{k + 1})}\left(\bm{c}^{(k)} - \bm{e}_t^{(k)}\right) \ge \deg^{(T \setminus v_{k + 1})}(v_t)$. By Proposition \ref{AdditivityOfFiring},
        \begin{align*}
            \bm{e}_t^{(k)} \cdot \bm{F}_{\mathbb{N}_{[1, t - 1]}}^{(T \setminus v_{k + 1})}\left(\bm{c}^{(k)} - \bm{e}_t^{(k)}\right) = \bm{e}_t^{(k)} \cdot \bm{F}_{\mathbb{N}_{[1, t - 1]}}^{(T \setminus v_{k + 1})}\left(\bm{c}^{(k)}\right) - 1.
        \end{align*}
        Hence,
        \begin{align*}
            \bm{e}_t^{(k)} \cdot \bm{F}_{\mathbb{N}_{[1, t - 1]}}^{(T \setminus v_{k + 1})}\left(\bm{c}^{(k)}\right) \ge \deg^{(T \setminus v_{k + 1})}(v_t) + 1 = \deg^{(T)}(v_t)
        \end{align*}
        Now observe that
        \begin{align*}
            \bm{F}_{\mathbb{N}_{[1, t - 1]}}^{(T)}\left(\left(\bm{c}^{(k)}, c_{k + 1}\right)\right) = \left(\bm{F}_{\mathbb{N}_{[1, t - 1]}}^{(T \setminus v_{k + 1})}\left(\bm{c}^{(k)}\right), c_{k + 1}\right)
        \end{align*}
        since $v_t$ never fires in the firing sequence $\mathbb{N}_{[1, t - 1]}$. Thus,
        \begin{align*}
            \bm{e}_t^{(k)} \cdot \bm{F}_{\mathbb{N}_{[1, t - 1]}}^{(T)}\left(\left(\bm{c}^{(k)}, c_{k + 1}\right)\right) \ge \deg^{(T)}(v_t),
        \end{align*}
        which proves that the firing of $v_t$ in $\Psi$ is legal. In addition,
        \begin{align*}    
            \bm{F}_{\mathbb{N}_{[1, t]}}^{(T)}\left(\left(\bm{c}^{(k)}, c_{k + 1}\right)\right) = \left(\bm{F}_{\mathbb{N}_{[1, t]}}^{(T \setminus v_{k + 1})}\left(\bm{c}^{(k)} - \bm{e}_t^{(k)}\right), c_{k + 1} + 1\right)
        \end{align*}
        since the firing of $v_t$ on $T$ causes $v_t$ to lose an additional chip to $v_{k + 1}$ compared to the firing of $v_t$ on $T \setminus v_{k + 1}$. Due to the legality of $\Phi$, the firing sequence $\mathbb{N}_{t + 1}^{k}$ is legal on $T \setminus v_{k + 1}$ starting from $\bm{F}_{\mathbb{N}_{[1, t]}}^{(T \setminus v_{k + 1})}\left(\bm{c}^{(k)} - \bm{e}_t^{(k)}\right)$. Since $v_t$ never fires in the firing sequence $\mathbb{N}_{t + 1}^{k}$, $\mathbb{N}_{t + 1}^{k}$ must also be a legal firing sequence on $T$ starting from $\left(\bm{F}_{\mathbb{N}_{[1, t]}}^{(T \setminus v_{k + 1})}\left(\bm{c}^{(k)} - \bm{e}_t^{(k)}\right), c_{k + 1} + 1\right)$. Thus, $\mathbb{N}_{t + 1}^k$ is a legal firing sequence on $T$ starting from $\bm{F}_{\mathbb{N}_{[1, t]}}^{(T)}\left(\left(\bm{c}^{(k)}, c_{k + 1}\right)\right)$, which proves that every firing move in $\Psi$ is legal, with the possible exception of the last one, the firing of $v_{k + 1}$. But the firing of $v_{k + 1}$ in $\Psi$ is also legal because it yields the legal chip configuration $\left(\bm{c}^{(k)}, c_{k + 1}\right)$. Thus, we have shown that all firing moves in $\Psi$ are legal, which completes this case and proves the desired claim.

\end{proof}



The result of Theorem \ref{srcsubtreeresult} yields the following corollaries for trees.

\begin{corollary}\label{subtreeconfigsaresrcs}
    Let $T$ be an $n$-vertex tree, and let $\bm{s}^{(n)}$ be a self-reachable configuration on $T$. Then the chip configuration formed by $\bm{s}^{(n)}$ on any subtree $T^*$ of $T$ must be self-reachable on $T^*$.
\end{corollary}

\begin{proof}
    If $\bm{s}^{(n)}$ does not form a self-reachable configuration on some subtree $T^*$ of $T$, then $\bm{s}^{(n)}$ does not form a self-reachable configuration on $T$ by Theorem \ref{srcsubtreeresult}. 
\end{proof}

\begin{corollary}\label{AddChipToLeafOrNeighbor}
    Let $T$ be an $n$-vertex tree ($n \in \mathbb{N}_2$) where $v_n$ is a leaf connected to $v_{n - 1}$. Then the following are equivalent:

    \begin{enumerate}
        \item $\bm{s}^{(n - 1)}$ is a self-reachable configuration on $T \setminus v_{n}$.

        \item $\left(\bm{s}^{(n - 1)} + \bm{e}_{n - 1}^{(n - 1)}, 0\right)$ is a self-reachable configuration on $T$.

        \item $\left(\bm{s}^{(n - 1)}, 1\right)$ is a self-reachable configuration on $T$.
    \end{enumerate}
\end{corollary}

\begin{proof}
    Note that (3) implies (2) by 
    Lemma \ref{closureofsrcs}
    because the configuration $\left(\bm{s}^{(n - 1)} + \bm{e}_{n - 1}^{(n - 1)}, 0\right)$ can be reached via legal firing moves from $\left(\bm{s}^{(n - 1)}, 1\right)$. Also note that (3) implies (1) by Corollary \ref{subtreeconfigsaresrcs}. We will next show that (1) implies (2). Suppose towards contradiction that $\left(\bm{s}^{(n - 1)} + \bm{e}_{n - 1}^{(n - 1)}, 0\right)$ is not self-reachable on $T$. Then there exists an $m$-vertex subtree $T^*$ of $T$ on which $\left(\bm{s}^{(n - 1)} + \bm{e}_{n - 1}^{(n - 1)}, 0\right)$ has fewer than $m - 1$ chips. If $T^*$ does not contain $v_n$, then $\bm{s}^{(n - 1)} + \bm{e}_{n - 1}^{(n - 1)}$, and by extension $\bm{s}^{(n - 1)}$, contains fewer than $m - 1$ chips on an $m$-vertex subtree of $T \setminus v_n$, which contradicts Theorem \ref{srcsubtreeresult}. Therefore, $T^*$ must contain $v_n$. Note that since $\left(\bm{s}^{(n - 1)} + \bm{e}_{n - 1}^{(n - 1)}, 0\right)$ contains fewer than $m - 1$ chips on $T^*$, $\bm{s}^{(n - 1)}$ contains fewer than $m - 2$ chips on $T^* \setminus v_n$, an $(m - 1)$-vertex subtree of $T \setminus v_n$. This contradicts Theorem \ref{srcsubtreeresult}. A very similar subtree argument shows that (1) implies (3), which completes the proof.
\end{proof}

We now turn our attention to proving our second main result of the section, Theorem \ref{srcmutualreachability}.
Before we provide this proof, we must consider the following lemma for self-reachable configurations on graphs.

\begin{lemma}\label{everyvertexcaneventuallyfire}
    Let $G$ be a connected simple $n$-vertex graph, and let $\bm{s}^{(n)}$ be a self-reachable chip configuration on $T$. Then for each $i \in \mathbb{N}_{[1, n]}$, there exists a chip configuration $\bm{c}_i^{(n)} \in R^{(G)}\left(\bm{s}^{(n)}\right)$ such that it is legal to fire $v_i$ on $G$ starting from $\bm{c}_i^{(n)}$.
\end{lemma}

\begin{proof}
    We know by Proposition \ref{srcfireeachvertexonce} that there exists a legal $n$-vertex firing sequence $\Phi$ that fires each vertex of $G$ exactly once:
    \begin{equation}
        \bm{F}_{\Phi}^{(G)}\left(\bm{s}^{(n)}\right) = \bm{s}^{(n)}.\notag
    \end{equation}
    Without loss of generality, we let $\Phi = \{1, 2, \ldots, n\}$. Then $v_1$ can be legally fired starting from $\bm{s}^{(n)} \in R^{(G)}\left(\bm{s}^{(n)}\right)$, $v_2$ can be legally fired starting from $\bm{F}_{\{1\}}\left(\bm{s}^{(n)}\right) \in R^{(G)}\left(\bm{s}^{(n)}\right)$, and so on, and $v_n$ can be legally fired starting from $\bm{F}_{\mathbb{N}_{[1, n - 1]}}\left(\bm{s}^{(n)}\right) \in R^{(G)}\left(\bm{s}^{(n)}\right)$. This proves the claim.
\end{proof}

\begin{proof}[Proof of Theorem \ref{srcmutualreachability}]
   Lemma \ref{closureofsrcs} 
   proves sufficiency, so we only need to prove necessity, which we do by induction on $n$. 
   The base case of $n=1$ vacuously holds.
   For induction hypothesis, assume that if $T$ is a tree with $k$ vertices, $\bm{s}^{(k)}$ is a self-reachable chip configuration on $T$, and $\bm{c}^{(k)}$ is a self-reachable chip configuration on $T$ with the same number of chips as $\bm{s}^{(k)}$, then $\bm{c}^{(k)} \in R^{(T)}\left(\bm{s}^{(k)}\right)$. We now let $T$ be a $(k + 1)$-vertex tree, and we let $v_{k + 1}$ be a vertex of degree 1 connected to $v_k$. Now let $\left(\bm{s}^{(k - 1)}, s_k, s_{k + 1}\right)$ and $\left(\bm{c}^{(k - 1)}, c_k, c_{k + 1}\right)$ be self-reachable configurations on $T$ with $\ell \in \mathbb{N}_k$ chips. Without loss of generality, suppose $s_{k + 1} \ge c_{k + 1}$. By Lemma \ref{closureofsrcs}, 
    it suffices to prove that one of $\left(\bm{c}^{(k - 1)}, c_k, c_{k + 1}\right)$ and $\left(\bm{s}^{(k - 1)}, s_k, s_{k + 1}\right)$ is reachable from the other. To do this, we consider three cases and construct a way to legally reach $\left(\bm{c}^{(k - 1)}, c_k, c_{k + 1}\right)$ from $\left(\bm{s}^{(k - 1)}, s_k, s_{k + 1}\right)$ in each case.
    
 \textbf{Case 1:}  $s_{k + 1} \ge c_{k + 1} > 0$:
Start with $\left(\bm{s}^{(k - 1)}, s_k, s_{k + 1}\right)$ and begin by firing $v_{k + 1}$ until the configuration $\left(\bm{s}^{(k - 1)}, s_k + s_{k + 1} - c_{k + 1}, c_{k + 1}\right)$ is reached. 
We know that this will eventually occur because $s_{k + 1} \ge c_{k + 1}$.
Next, we know by 
			Lemma \ref{closureofsrcs}            
            that  $\left(\bm{s}^{(k - 1)}, s_k + s_{k + 1} - c_{k + 1}, c_{k + 1}\right)$ must be self-reachable on $T$. Thus, Corollary \ref{subtreeconfigsaresrcs} implies that $\left(\bm{s}^{(k - 1)}, s_k + s_{k + 1} - c_{k + 1}\right)$ is self-reachable on $T \setminus v_{k + 1}$. Also note that both $\left(\bm{s}^{(k - 1)}, s_k + s_{k + 1} - c_{k + 1}\right)$ and $\left(\bm{c}^{(k - 1)}, c_k\right)$ are self-reachable configurations on $T$ with $\ell - c_{k + 1}$ chips. Therefore, by our induction hypothesis, it is possible to reach $\left(\bm{c}^{(k - 1)}, c_k\right)$ from $\left(\bm{s}^{(k - 1)}, s_k + s_{k + 1} - c_{k + 1}\right)$ via a sequence of legal firing moves on $T \setminus v_{k + 1}$. We perform the same sequence of firing moves on $T$, except we replace each firing of $v_k$ on $T \setminus v_{k + 1}$ with a firing of $v_{k + 1}$ followed immediately by a firing of $v_k$ on $T$. By Lemma \ref{treeinductionbasis}, this is legal and yields the configuration $\left(\bm{c}^{(k - 1)}, c_k, c_{k + 1}\right)$ on $T$, as desired.

 \textbf{Case 2:} $s_{k + 1} > c_{k + 1} = 0$:
  We know by Lemma \ref{everyvertexcaneventuallyfire} that there exists $\left(\bm{d}^{(k - 1)}, d_k, d_{k + 1}\right) \in R^{(T)}\left(\left(\bm{c}^{(k - 1)}, c_k, c_{k + 1}\right)\right)$ such that it is legal to fire the vertex $v_{k + 1}$ on $T$ starting from $\left(\bm{d}^{(k - 1)}, d_k, d_{k + 1}\right)$, meaning that $d_{k + 1} \ge 1$. We can then use the same reasoning as in the previous case to deduce that one of $\left(\bm{d}^{(k - 1)}, d_k, d_{k + 1}\right)$ and $\left(\bm{s}^{(k - 1)}, s_k, s_{k + 1}\right)$ will be reachable from the other. So by Lemma \ref{closureofsrcs},
        $\left(\bm{s}^{(k - 1)}, s_k, s_{k + 1}\right) \in R^{(T)}\left(\left(\bm{d}^{(k - 1)}, d_k, d_{k + 1}\right)\right)$, which means that $\left(\bm{s}^{(k - 1)}, s_k, s_{k + 1}\right) \in R^{(T)}\left(\left(\bm{c}^{(k - 1)}, c_k, c_{k + 1}\right)\right)$, as desired.
        
 \textbf{Case 3:} $s_{k + 1} = c_{k + 1} = 0$:
 As in the previous case, we can deduce the existence of 
        \begin{align*}
            &\left(\bm{d}^{(k - 1)}, d_k, d_{k + 1}\right) \in R^{(T)}\left(\left(\bm{c}^{(k - 1)}, c_k, c_{k + 1}\right)\right), \\
            &\left(\bm{\sigma}^{(k - 1)}, \sigma_k, \sigma_{k + 1}\right) \in R^{(T)}\left(\left(\bm{s}^{(k - 1)}, s_k, s_{k + 1}\right)\right)
        \end{align*}
        such that $d_{k + 1} \ge 1$ and $\sigma_{k + 1} \ge 1$. By the same argument as in Case 1, one of $\left(\bm{d}^{(k - 1)}, d_k, d_{k + 1}\right)$ and $\left(\bm{\sigma}^{(k - 1)}, \sigma_k, \sigma_{k + 1}\right)$ is reachable from the other. Then by Lemma \ref{closureofsrcs}, 
        we see that one of $\left(\bm{c}^{(k - 1)}, c_k, c_{k + 1}\right)$ and $\left(\bm{s}^{(k - 1)}, s_k, s_{k + 1}\right)$ is reachable from the other.
    
    
\end{proof}    

An immediate result of Theorem \ref{srcmutualreachability} is the following observation on the equivalence classes of self-reachable configurations on tress. 

\begin{corollary}
    Let $T$ be an $n$-vertex tree. The equivalence relation $\sim$ on $S_\ell^{(T)}$, as defined in Corollary \ref{srcequivalencerelation}, has only one equivalence class, $S_\ell^{(T)}$ itself.
\end{corollary}


\section{An enumeration result}\label{sec:combinatorial}

We conclude this paper with an interesting enumerative result for the number of self-reachable configurations on trees. 

\begin{theorem}\label{srccountingtheorem}
    Let $T$ be an $n$-vertex tree and let $\displaystyle C_{\ell,n}\coloneq \left|S_\ell^{(T)}\right|$.    
    Then $C_{\ell,n}$ satisfies  recurrence relation for $\ell \in \mathbb{N}_2$ and $n \in \mathbb{N}_2$:
    \begin{align}\label{eq4.1}
        C_{\ell, n} = C_{\ell - 1, n} + 2C_{\ell - 1, n - 1} - C_{\ell - 2, n - 1}.
    \end{align}
    with $C_{\ell, 1} = 1$ for all $\ell \in \mathbb{N}_0$, $C_{0, n} = 0$ for all $n \in \mathbb{N}_2$, $C_{1, 2} = 2$, and $C_{1, n} = 0$ for all $n \in \mathbb{N}_3$. 
\end{theorem}

\begin{remark}
It is worth noting that the sequence from Theorem \ref{srccountingtheorem} appears in the OEIS as sequence A049600 \cite{oeisA049600}. 
\end{remark}

In order to prove this result, we must first state and prove a utility proposition.

\begin{proposition}\label{Add Chips to SRC}
    Let $G$ be a connected simple $n$-vertex graph, and let $\bm{s}^{(n)}$ be a self-reachable configuration on $G$. Let $\bm{c}^{(n)}$ be a chip configuration on $G$ such that $\bm{c}^{(n)}$ has at least as many chips on each vertex of $G$ as $\bm{s}^{(n)}$. Then $\bm{c}^{(n)}$ is self-reachable on $G$.
\end{proposition}

\begin{proof}
    By Proposition \ref{srcfireeachvertexonce}, there exists a legal firing sequence $\Phi$ on $G$ starting from $\bm{s}^{(n)}$ such that $\bm{F}_{\Phi}^{(G)}\left(\bm{s}^{(n)}\right) = \bm{s}^{(n)}$. By Proposition \ref{LegalityProp}, since $\bm{c}^{(n)}$ has at least as many chips on each vertex of $G$ as $\bm{s}^{(n)}$, $\Phi$ will also be a legal firing sequence on $G$ starting from $\bm{c}^{(n)}$. Finally, by Lemma \ref{firingsequenceeachvertexonce}, $\bm{F}_{\Phi}^{(G)}\left(\bm{c}^{(n)}\right) = \bm{c}^{(n)}$. Thus, $\bm{c}^{(n)}$ is self-reachable on $G$, as claimed.
\end{proof}


\begin{proof}[Proof of Theorem \ref{srccountingtheorem}]
    The initial conditions follow as a tree with a single vertex has exactly one self-reachable configuration with $\ell$ chips,  a two-vertex tree with one chip has two possible self-reachable configurations, and Theorem \ref{srcsubtreeresult} respectively. The initial conditions imply that $C_{\ell, n}$ is a well-defined quantity whenever $\ell \le 1$ or $n \le 1$, and as such  $C_{\ell, n}$ depends only on $\ell$ and $n$ and not on the structure of $T$. It now suffices to prove that  \eqref{eq4.1} holds for arbitrary but fixed $\ell \in \mathbb{N}_2$ and $n \in \mathbb{N}_2$, under the assumption that $C_{\ell - 1, n}$, $C_{\ell - 1, n - 1}$, and $C_{\ell - 2, n - 1}$ are well-defined quantities. 
    Without loss of generality, suppose that $v_n$ is a leaf connected to $v_{n - 1}$. 
    To count the number of self-reachable configurations $\left(\bm{s}^{(n - 1)}, s_n\right)$ on $T$ with $\ell$ chips, we will divide this set into three disjoint classes and enumerate each one.

	\textbf{Class 1:} $s_n = 0$: We will show that self-reachable configurations of this form are in  bijection with the self-reachable configurations on $T \setminus v_n$ with $\ell - 1$ chips. If $\left(\bm{s}^{(n - 1)}, s_n\right)$ is a class 1 self-reachable configuration on $T$ with $\ell$ chips, we have that $\bm{s}^{(n - 1)}$ is self-reachable on $T \setminus v_n$ by Corollary \ref{subtreeconfigsaresrcs}. Note that if $\bm{s}^{(n - 1)} - \bm{e}_{n - 1}^{(n - 1)}$ has less than $m - 1$ chips on an $m$-vertex subtree of $T \setminus v_n$, then this subtree must contain $v_{n - 1}$ or else $\bm{s}^{(n - 1)}$ would not be self-reachable on $T \setminus v_n$ by Theorem 4.7. But if we then add $v_n$ to this subtree and form an $(m + 1)$-vertex subtree of $T$, then $\left(\bm{s}^{(n - 1)}, s_n\right)$ has less than $m$ chips on an $(m + 1)$-vertex subtree of $T$, which contradicts the fact that $\left(\bm{s}^{(n - 1)}, s_n\right)$ is self-reachable on $T$. Thus, $\bm{s}^{(n - 1)} - \bm{e}_{n - 1}^{(n - 1)}$, which has $\ell - 1$ chips, is self-reachable on $T \setminus v_n$. This shows that for each distinct class 1 self-reachable configuration on $T$ with $\ell$ chips, there exists a distinct self-reachable configuration on $T \setminus v_n$ with $\ell - 1$ chips. On the other hand, if $\bm{\sigma}^{(n - 1)}$ is a self-reachable configuration on $T \setminus v_n$ with $\ell - 1$ chips, then $\left(\bm{\sigma}^{(n - 1)} + \bm{e}_{n - 1}^{(n - 1)}, 0\right)$ is a self-reachable configuration on $T$ with $\ell$ chips by Corollary \ref{AddChipToLeafOrNeighbor}. This shows that for each distinct self-reachable configuration on $T \setminus v_n$ with $\ell - 1$ chips, there exists a distinct class 1 self-reachable configuration on $T$ with $\ell$ chips.
	Thus, we have the desired bijection. 

        \textbf{Class 2:} $s_n = 1$: 
		We will show that self-reachable configurations of this form are also in bijection  with the self-reachable configurations on $T \setminus v_n$ with $\ell - 1$ chips. Note that if $\left(\bm{s}^{(n - 1)}, s_n\right)$ is a class 2 self-reachable configuration on $T$ with $\ell$ chips, then $\bm{s}^{(n - 1)}$ is self-reachable on $T \setminus v_n$ by Corollary \ref{subtreeconfigsaresrcs}, and $\bm{s}^{(n - 1)}$ has $\ell - 1$ chips on $T \setminus v_n$. This shows that for each distinct class 2 self-reachable configuration on $T$ with $\ell$ chips, there exists a distinct self-reachable configuration on $T \setminus v_n$ with $\ell - 1$ chips. On the other-hand, if $\bm{\sigma}^{(n - 1)}$ is a self-reachable configuration on $T \setminus v_n$ with $\ell - 1$ chips, then $\left(\bm{\sigma}^{(n - 1)}, 1\right)$ is a self-reachable configuration on $T$ with $\ell$ chips by Corollary \ref{AddChipToLeafOrNeighbor}. This shows that for each distinct self-reachable configuration on $T \setminus v_n$ with $\ell - 1$ chips, there exists a distinct class 2 self-reachable configuration on $T$ with $\ell$ chips. Thus, we have shown the desired bijection.

        \textbf{Class 3:} $s_n \ge 2$. We will show  that self-reachable configurations of this form are in bijection with the self-reachable configurations on $T$ with $\ell - 1$ chips that have at least one chip on $v_n$. We claim that if $\left(\bm{s}^{(n - 1)}, s_n\right)$ is a class 3 self-reachable configuration on $T$ with $\ell$ chips, then $\left(\bm{s}^{(n - 1)}, s_n - 1\right)$ is a self-reachable configuration on $T$ with $\ell - 1$ chips. Suppose towards contradiction that $\left(\bm{s}^{(n - 1)}, s_n - 1\right)$ is not self-reachable on $T$. Then by Theorem \ref{srcsubtreeresult}, there exists an $m$-vertex subtree $T^*$ of $T$ on which $\left(\bm{s}^{(n - 1)}, s_n - 1\right)$ has fewer than $m - 1$ chips. Without loss of generality, we can suppose that $T^*$ is chosen so that $m$ is minimal. If $T^*$ does not contain $v_n$, then $\left(\bm{s}^{(n - 1)}, s_n\right)$ would also have fewer than $m - 1$ chips on $T^*$, contradicting Theorem \ref{srcsubtreeresult}. So $T^*$ must contain $v_n$. Note that since $s_n - 1 \ge 1$, $\left(\bm{s}^{(n - 1)}, s_n - 1\right)$ must have fewer than $m - 2$ chips on the $(m - 1)$-vertex subtree $T^* \setminus v_n$. But this contradicts that fact that $T^*$ was chosen to minimize $m$. Thus, $\left(\bm{s}^{(n - 1)}, s_n - 1\right)$ is self-reachable on $T$. This shows that for each distinct class 3 self-reachable configuration on $T$ with $\ell$ chips, there exists a distinct self-reachable configuration on $T$ with $\ell - 1$ chips and at least one chip on $v_n$. On the other hand, if $\left(\bm{\sigma}^{(n - 1)}, \sigma_n\right)$ is a self-reachable configuration on $T$ with $\ell - 1$ chips and $\sigma_n \ge 1$, then $\left(\bm{\sigma}^{(n - 1)}, \sigma_n + 1\right)$ is a class 3 self-reachable configuration on $T$ with $\ell$ chips by Proposition \ref{Add Chips to SRC}. 
        Hence, we have shown the desired bijection.
   
   Using these bijections, we can enumerate $C_{\ell,n}$.
   Since $C_{\ell - 1, n - 1}$ is assumed to be a well-defined quantity, the number of self-reachable configurations on $T \setminus v_n$ with $\ell - 1$ chips is $C_{\ell - 1, n - 1}$.
   Therefore, the cardinality of both class 1 and class 2 is $C_{\ell-1,n-1}$.
    To determine the number class 3 self-reachable configurations, we must count the number of self-reachable configurations on $T$ with $\ell - 1$ chips that have at least one chip on $v_n$. 
    We can instead count how many self-reachable configurations on $T$ with $\ell - 1$ chips have no chips on $v_n$ and subtract this from the total number of self-reachable configuration with $\ell-1$ chips on $T$.
    Note that, since we are enumerating class 1 configurations with $\ell-1$ chips, reapplying the argument gives a bijection with self-reachable configurations $T \setminus v_n$ with $\ell - 2$ chips, which are enumerated by $C_{\ell - 2, n - 1}$ which is assumed to be well-defined.
    Therefore there are $C_{\ell - 1, n} - C_{\ell - 2, n - 1}$ class 3 self-reachable configurations.
    Summing across the three classes, we obtain the desired result 
    \begin{align*}
        C_{\ell, n} = C_{\ell - 1, n} + 2C_{\ell - 1, n - 1} - C_{\ell - 2, n - 1}.
    \end{align*}

\end{proof}

\end{document}